\documentclass[12pt,reqno]{article}
\usepackage{amssymb,amscd,amsmath,amsthm,color}
\newtheorem{thm}{Theorem}[section]
\newtheorem{prop}[thm]{Proposition}
\newtheorem{lemma}[thm]{Lemma}
\newtheorem{cor}[thm]{Corollary}

\numberwithin{equation}{section}

\def\cH{\mathcal{H}}
\def\dom{\mathrm{dom}\,}
\def\eps{\varepsilon}
\def\diag{\mathrm{diag}}
\def\Tr{\mathrm{Tr}\,}
\def\cI{\mathcal{I}}
\def\bC{\mathbb{C}}
\def\bR{\mathbb{R}}
\def\lin{\mathrm{span}}
\def\bv{\mathbf{v}}
\def\bw{\mathbf{w}}
\def\<{\langle}
\def\>{\rangle}
\def\Re{\mathrm{Re}\,}
\def\cO{\mathcal{O}}
\def\cU{\mathcal{U}}
\def\cV{\mathcal{V}}
\def\gap{\mathrm{gap}}

\begin{document}
\baselineskip=15pt
\allowdisplaybreaks

\title{Anti Lie-Trotter formula}

\author{Koenraad M.R.\ Audenaert$^{1,2,}$\footnote{E-mail: koenraad.audenaert@rhul.ac.uk}
\ and
Fumio Hiai$^{3,}$\footnote{E-mail: hiai.fumio@gmail.com}}

\date{\today}
\maketitle

\begin{center}
$^1$\,Department of Mathematics, Royal Holloway University of London, \\
Egham TW20 0EX, United Kingdom
\end{center}

\begin{center}
$^2$\,Department of Physics and Astronomy, Ghent University, \\
S9, Krijgslaan 281, B-9000 Ghent, Belgium
\end{center}

\begin{center}
$^3$\,Tohoku University (Emeritus), \\
Hakusan 3-8-16-303, Abiko 270-1154, Japan
\end{center}

\medskip
\begin{abstract}
\noindent
Let $A$ and $B$ be positive semidefinite matrices.
The limit of the expression $Z_p:=(A^{p/2}B^pA^{p/2})^{1/p}$ as $p$ tends to $0$ is given by the well known
Lie-Trotter-Kato formula. A similar formula holds for the limit of $G_p:=(A^p\,\#\,B^p)^{2/p}$ as $p$ tends to $0$, where $X\,\#\,Y$ is the geometric mean of $X$ and $Y$.
In this paper we study the complementary limit of $Z_p$ and $G_p$ as $p$ tends to $\infty$, with the ultimate goal of finding an explicit formula, which we
call the anti Lie-Trotter formula.
We show that the limit of $Z_p$ exists and find an explicit formula in a special case. The limit of $G_p$ is shown for $2\times2$ matrices only.


\end{abstract}

\section{Introduction}

When $H,K$ are lower bounded self-adjoint operators on a Hilbert space $\cH$ and $H_+,K_+$
are their positive parts, the sum of $H$ and $K$ can be given a precise meaning as a lower bounded self-adjoint
operator on the subspace $\cH_0$, which is defined as the closure of $\dom H_+^{1/2}\cap\dom K_+^{1/2}$.
We denote this formal sum as $H\dot+K$.
Then the well-known Lie-Trotter-Kato product formula, as originally established in \cite{Tr,Ka78} and
refined by many authors, expresses the convergence
$$
\lim_{n\to\infty}(e^{-tH/n}e^{-tK/n})^n=e^{-t(H\dot+K)}P_0,\qquad t>0,
$$
in the strong operator topology (uniformly in $t\in[a,b]$ for any $0<a<b$), where $P_0$ is
the orthogonal projection onto $\cH_0$. Although this formula is usually stated for densely-defined
$H,K$, the proof in \cite{Ka78} applies to
the improper case (i.e., $H,K$ are not densely-defined) as well,
under the convention that $e^{-tH}=0$ on $(\dom H)^\perp$ for $t>0$, and similarly for $e^{-tK}$.

The Lie-Trotter-Kato formula can easily be modified to
symmetric form and with a continuous parameter as \cite[Theorem 3.6]{Hi1}
$$
\lim_{p\searrow0}(e^{-ptH/2}e^{-ptK}e^{-ptH/2})^{1/p}=e^{-t(H\dot+K)}P_0,\qquad t>0.
$$
When restricted to matrices (and to $t=1$) this can be rephrased as
\begin{equation}\label{F-1.1}
\lim_{p\searrow0}(A^{p/2}B^pA^{p/2})^{1/p}=P_0\exp(\log A\dot+\log B),
\end{equation}
where $A$ and $B$ are positive semidefinite matrices (written as $A,B\ge0$ below),
$P_0$ is now the orthogonal projection onto the intersection of the supports of $A,B$ and
$\log A\dot+\log B$ is defined as $P_0(\log A)P_0+P_0(\log B)P_0$.

When $\sigma$ is an operator mean \cite{KA} corresponding to an operator monotone function
$f$ on $(0,\infty)$ such that $\alpha:=f'(1)$ is in $(0,1)$, the operator mean version of the
Lie-Trotter-Kato product formula is the convergence \cite[Theorem 4.11]{Hi1}
$$
\lim_{p\searrow0}(e^{-ptH}\,\sigma\,e^{-ptK})^{1/p}
=e^{-t((1-\alpha)H\dot+\alpha K)},\qquad t>0,
$$
in the strong operator topology, for a bounded self-adjoint operator $H$ and a lower-bounded
self-adjoint operator $K$ on $\cH$. Although it is not known whether the above formula holds
even when both $H,K$ are lower bounded (and unbounded), we can verify that \eqref{F-1.1} has the operator mean
version
\begin{equation}\label{F-1.2}
\lim_{p\searrow0}(A^p\,\sigma\,B^p)^{1/p}=P_0\exp((1-\alpha)\log A\dot+\alpha\log B),
\end{equation}
for matrices $A,B\ge0$.  A proof of \eqref{F-1.2} is supplied in an appendix of this paper since it is
not our main theme.

In particular, let $\sigma$ be the geometric mean $A\,\#\,B$ (introduced first in \cite{PW} and
further discussed in \cite{KA}), corresponding to the operator monotone function $f(x)=x^{1/2}$
(hence $\alpha=1/2$). Then \eqref{F-1.2} yields
\begin{equation}\label{F-1.3}
\lim_{p\searrow0}(A^p\,\#\,B^p)^{2/p}=P_0\exp(\log A\dot+\log B),
\end{equation}
which has the same right-hand side as \eqref{F-1.1}.

It turns out that the convergence of both \eqref{F-1.1} and \eqref{F-1.3} is monotone in the log-majorization order.
For $d\times d$ matrices $X,Y\ge0$, the log-majorization relation
$X\prec_{(\log)}Y$ means that
$$
\prod_{i=1}^k\lambda_i(X)\le\prod_{i=1}^k\lambda_i(Y),\qquad1\le k\le d,
$$
with equality for $k=d$, where $\lambda_1(X)\ge\dots\ge\lambda_d(X)$ are the eigenvalues of
$X$ sorted in decreasing order and counting multiplicities.
The Araki-Lieb-Thirring inequality can be written in terms of log-majorization as
\begin{equation}\label{F-1.4}
(A^{p/2}B^pA^{p/2})^{1/p}\prec_{(\log)}(A^{q/2}B^qA^{q/2})^{1/q}
\quad\mbox{if}\quad0<p<q,
\end{equation}
for matrices $A,B\ge0$, see \cite{LT,Ar,AH}. One can also consider the complementary version of
\eqref{F-1.4} in terms of the geometric mean. Indeed, for $A,B\ge0$ we have \cite{AH}
\begin{equation}\label{F-1.5}
(A^q\,\#\,B^q)^{2/q}\prec_{(\log)}(A^p\,\#\,B^p)^{2/p}
\quad\mbox{if}\quad0<p<q.
\end{equation}
Hence, for matrices $A,B\ge0$, we see that
$Z_p:=(A^{p/2}B^pA^{p/2})^{1/p}$ and $G_p:=(A^p\,\#\,B^p)^{2/p}$ both tend to $P_0\exp(\log A\dot+\log B)$ as $p\searrow0$,
with the former decreasing (by \eqref{F-1.4}) and
the latter increasing (by \eqref{F-1.5}) in the log-majorization order.

The main topic of this paper is the complementary question about what happens to the limits of $Z_p$ and $G_p$ as $p$ tends to $\infty$ instead of $0$. Although
this seems a natural mathematical problem, we have not been able to find an explicit
statement of concern in the literature. It is obvious that if $A$ and $B$ are commuting then $G_p=AB=Z_p$,
independently of $p>0$. However, if $A$ and $B$ are not commuting, then the limit behavior of
$Z_p$ and its eigenvalues as $p\to\infty$ is of a rather complicated combinatorial nature,
and that of $G_p$ seems even more complicated.

The problem of finding an explicit formula, which we henceforth call the anti Lie-Trotter formula, also emerges from recent
developments of new R\'enyi relative entropies relevant to quantum information theory.
Indeed, the recent paper \cite{AD} proposed to generalize the R\'enyi relative entropy as
$$
D_{\alpha,z}(\rho\|\sigma):={1\over\alpha-1}\log\Tr
\bigl(\rho^{\alpha/2z}\sigma^{(1-\alpha)/z}\rho^{\alpha/2z}\bigr)^z
$$
for density matrices $\rho,\sigma$ with two real parameters $\alpha,z$, and discussed the
limit formulas when $\alpha,z$ converge to some special values. The limit case of
$D_{\alpha,z}(\rho\|\sigma)$ as $z\to0$ with $\alpha$ fixed is exactly related to our
anti Lie-Trotter problem.

The rest of the paper is organized as follows. In Section \ref{sec2} we prove the existence of the
limit of $Z_p$ as $p\to\infty$ when $A,B$ are $d\times d$ positive semidefinite matrices.
In Section \ref{sec3} we analyze the case when the limit eigenvalue list of $Z_p$ becomes
$\lambda_i(A)\lambda_i(B)$ ($1\le i\le d$), the maximal case in the log-majorization order.
In Section \ref{sec4} we extend the existence of the limit of $Z_p$ to that of
$\bigl(A_1^{p/2}\cdots A_{m-1}^{p/2}A_m^pA_{m-1}^{p/2}\cdots A_1^{p/2}\bigr)^{1/p}$
with more than two matrices. Finally in Section \ref{sec5} we treat $G_p$; however we can prove
the existence of the limit of $G_p$ as $p\to\infty$ only when $A,B$ are $2\times2$
matrices, and the general case must be left unsettled. The paper contains two appendices.
The first is a proof of a technical lemma stated in Section \ref{sec2}, and the second supplies
the detailed proof of \eqref{F-1.2}.
\section{Limit of $(A^{p/2}B^pA^{p/2})^{1/p}$ as $p\to\infty$ \label{sec2}}

Let $A$ and $B$ be $d\times d$ positive semidefinite matrices having the eigenvalues
$a_1\ge\cdots\ge a_d\,(\ge0)$ and $b_1\ge\cdots\ge b_d\,(\ge0)$, respectively, sorted in decreasing
order and counting multiplicities. Let $\{v_1,\dots,v_d\}$ be an orthonormal set of eigenvectors
of $A$ such that $Av_i=a_iv_i$ for $i=1,\dots,d$, and $\{w_1,\dots,w_d\}$ an orthonormal
set of eigenvectors of $B$ in a similar way. Then $A$ and $B$ are diagonalized as
\begin{align}
A&=V\diag(a_1,\dots,a_d)V^*=\sum_{i=1}^da_iv_iv_i^*, \label{F-2.1}\\
B&=W\diag(b_1,\dots,b_d)W^*=\sum_{i=1}^db_iw_iw_i^*. \label{F-2.2}
\end{align}

For each $p>0$ define a positive semidefinite matrix
\begin{equation}\label{F-2.3}
Z_p:=(A^{p/2}B^pA^{p/2})^{1/p},
\end{equation}
whose eigenvalues are denoted as $\lambda_1(p)\ge\lambda_2(p)\ge\dots\ge\lambda_d(p)$, again in
decreasing order and counting multiplicities.

\begin{lemma}\label{L-2.1}
For every $i=1,\dots,d$ the limit
\begin{equation}\label{F-2.4}
\lambda_i:=\lim_{p\to\infty}\lambda_i(p)
\end{equation}
exists, and $a_1b_1\ge\lambda_1\dots\ge\lambda_d\ge a_db_d$.
\end{lemma}

\begin{proof}
Since $(a_1b_1)^pI\ge A^{p/2}B^pA^{p/2}\ge(a_db_d)^pI$, we have
$a_1b_1\ge\lambda_i(p)\ge a_db_d$ for all $i=1,\dots,d$ and all $p>0$. By the
Araki-Lieb-Thirring inequality \cite{Ar} (or the log-majorization \cite{AH}), for every
$k=1,\dots,d$ we have
\begin{equation}\label{F-2.5}
\prod_{i=1}^k\lambda_i(p)\le\prod_{i=1}^k\lambda_i(q)\quad
\mbox{if}\quad0<p<q.
\end{equation}
Therefore, the limit $\eta_k$ of $\prod_{i=1}^k\lambda_i(p)$ as $p\to\infty$ exists for any
$k=1,\dots,d$ so that $\eta_1\ge\dots\ge\eta_d\ge0$. Let $m$ ($0\le m\le d$) be the
largest $k$ such that $\eta_k>0$ (with $m:=0$ if $\eta_1=0$). When $1\le k\le m$, we have
$\lambda_k(p)\to\eta_k/\eta_{k-1}$ (where $\eta_0:=1$) as $p\to\infty$. When $m<d$,
$\lambda_{m+1}(p)\to\eta_{m+1}/\eta_m=0$ as $p\to\infty$. Hence $\lambda_k(p)\to0$ for all
$k>m$. Therefore, the limit of $\lambda_i(p)$ as $p\to\infty$ exists for any $i=1,\dots,d$.
The latter assertion is clear now.
\end{proof}

\begin{lemma}\label{L-2.2}
The first eigenvalue in \eqref{F-2.4} is given by
$$
\lambda_1=\max\{a_ib_j:(V^*W)_{ij}\ne0\},
$$
where $(V^*W)_{ij}$ denotes the $(i,j)$ entry of $V^*W$.
\end{lemma}

\begin{proof}
Write $V^*W=[u_{ij}]$. We observe that
$$
\bigl(V^*A^{p/2}B^pA^{p/2}V\bigr)_{ij}
=\sum_{k=1}^du_{ik}\overline u_{jk}a_i^{p/2}a_j^{p/2}b_k^p.
$$
In particular,
$$
\bigl(V^*A^{p/2}B^pA^{p/2}V\bigr)_{ii}
=\sum_{k=1}^d|u_{ik}|^2a_i^pb_k^p
$$
and hence we have
$$
\lambda_1(p)^p\le\Tr A^{p/2}B^pA^{p/2}
=\sum_{i=1}^d\sum_{k=1}^d|u_{ik}|^2a_i^pb_k^p
\le d^2\max\{a_i^pb_k^p:u_{ik}\ne0\},
$$
where $\Tr$ is the usual trace functional on $d\times d$ matrices. Therefore,
\begin{equation}\label{F-2.6}
\lambda_1(p)\le d^{2/p}\max\{a_ib_k:u_{ik}\ne0\}.
\end{equation}
On the other hand, we have
$$
d\lambda_1(p)^p\ge\Tr A^{p/2}B^pA^{p/2}
\ge\min\{|u_{ik}|^2:u_{ik}\ne0\}\max\{a_i^pb_k^p:u_{ik}\ne0\}
$$
so that
\begin{equation}\label{F-2.7}
\lambda_1(p)\ge\biggl({\min\{|u_{ik}|^2:u_{ik}\ne0\}\over d}\biggr)^{1/p}
\max\{a_ib_k:u_{ik}\ne0\}.
\end{equation}
Estimates \eqref{F-2.6} and \eqref{F-2.7} give the desired expression immediately.
In fact, they prove the existence of the limit in \eqref{F-2.4} as well apart from
Lemma \ref{L-2.1}.
\end{proof}

In what follows, for each $k=1,\dots,d$ we write $\cI_d(k)$ for the set of all subsets
$I$ of $\{1,\dots,d\}$ with $|I|=k$. For $I,J\in\cI_d(k)$ we denote by $(V^*W)_{I,J}$
the $k\times k$ submatrix of $V^*W$ corresponding to rows in $I$ and columns in $J$;
hence $\det(V^*W)_{I,J}$ denotes the corresponding minor of $V^*W$. We also write
$a_I:=\prod_{i\in I}a_i$ and $b_I:=\prod_{i\in I}b_i$. Since $\det(V^*W)\ne0$, note that
for any $k=1,\dots,d$ and any $I\in\cI_d(k)$ we have $\det(V^*W)_{I,J}\ne0$ for some
$J\in\cI_d(k)$, and that for any $J\in\cI_d(k)$ we have $\det(V^*W)_{I,J}\ne0$ for some
$I\in\cI_d(k)$.

\begin{lemma}\label{L-2.3}
For every $k=1,\dots,d$,
\begin{equation}\label{F-2.8}
\lambda_1\lambda_2\cdots\lambda_k=\max\{a_Ib_J:I,J\in\cI_d(k),\,\det(V^*W)_{I,J}\ne0\}.
\end{equation}
\end{lemma}

\begin{proof}
For each $k=1,\dots,d$ the antisymmetric tensor powers $A^{\wedge k}$ and $B^{\wedge k}$
(see \cite{Bh1}) are given in the form of diagonalizations as
$$
A^{\wedge k}=V^{\wedge k}\diag(a_I)_{I\in\cI_d(k)}V^{\wedge k},\qquad
B^{\wedge k}=W^{\wedge k}\diag(b_I)_{I\in\cI_d(k)}W^{\wedge k},
$$
and the corresponding representation of the ${n\choose k}\times{n\choose k}$ unitary
matrix $V^{*\wedge k}W^{\wedge k}$ is given by
$$
(V^{*\wedge k}W^{\wedge k})_{I,J}=\det(V^*W)_{I,J},\qquad I,J\in\cI_d(k).
$$
Note that the largest eigenvalue of
$$
\bigl((A^{\wedge k})^{p/2}(B^{\wedge k})^p(A^{\wedge k})^{p/2}\bigr)^{1/p}
=\bigl((A^{p/2}B^pA^{p/2})^{1/p}\bigr)^{\wedge k}
$$
is $\lambda_1(p)\lambda_2(p)\cdots\lambda_k(p)$, whose limit as $p\to\infty$ is
$\lambda_1\lambda_2\cdots\lambda_k$ by Lemma \ref{L-2.1}. Apply Lemma \ref{L-2.2} to
$A^{\wedge k}$ and $B^{\wedge k}$ to obtain expression \eqref{F-2.8}.
\end{proof}

Let $\cH$ be a $d$-dimensional Hilbert space (say, $\bC^d$), $k$ be an integer with
$1\le k\le d$, and $\cH^{\wedge k}$ be the $k$-fold antisymmetric tensor of $\cH$. We
write $x_1\wedge\cdots\wedge x_k$ ($\in\cH^{\wedge k}$) for the antisymmetric tensor of
$x_1,\dots,x_k\in\cH$ (see \cite{Bh1}). The next lemma says that the Grassmannian
manifold $G(k,d)$ is realized in the projective space of $\cH^{\wedge k}$. Although the
lemma might be known to specialists, we cannot find a precise explanation in the literature.
So, for the convenience of the reader, we will present its sketchy proof in Appendix A
based on \cite{FGP}.

\begin{lemma}\label{L-2.4}
There are constants $\alpha,\beta>0$ (depending on only $d$ and $k$) such that
$$
\alpha\|P-Q\|\le\inf_{\theta\in\bR}
\|u_1\wedge\cdots\wedge u_k-e^{\sqrt{-1}\theta}v_1\wedge\cdots\wedge v_k\|
\le\beta\|P-Q\|
$$
for all orthonormal sets $\{u_1,\dots,u_k\}$ and $\{v_1,\dots,v_k\}$ and the respective
orthogonal projections $P$ and $Q$ onto $\lin\{u_1,\dots,u_k\}$ and $\lin\{v_1,\dots,v_k\}$,
where $\|P-Q\|$ is the operator norm of $P-Q$ and $\|\cdot\|$ inside infimum is the norm on
$\cH^{\wedge k}$.
\end{lemma}

The main result of the paper is the next theorem showing the existence of limit for the
anti version of \eqref{F-1.1}.

\begin{thm}\label{T-2.5}
For every $d\times d$ positive semidefinite matrices $A$ and $B$ the matrix $Z_p$ in
\eqref{F-2.3} converges as $p\to\infty$ to a positive semidefinite matrix.
\end{thm}

\begin{proof}
By replacing $A$ and $B$ with $VAV^*$ and $VBV^*$, respectively, we may assume that $V=I$
and so
$$
A=\diag(a_1,\dots,a_d),\qquad B=W\diag(b_1,\dots,b_d)W^*.
$$
Choose an orthonormal basis $\{u_1(p),\dots,u_d(p)\}$ of $\bC^d$ for which we have
$Z_pu_i(p)=\lambda_i(p)u_i(p)$ for $1\le i\le d$. Let $\lambda_i$ be given in Lemma
\ref{L-2.1}, and assume that $1\le k<d$ and $\lambda_1\ge\dots\ge\lambda_k>\lambda_{k+1}$.
Moreover, let $\lambda_1(Z_p^{\wedge k})\ge\lambda_2(Z_p^{\wedge k})\ge\dots$ be the
eigenvalues of $Z_p^{\wedge k}$ in decreasing order. We note that
\begin{align}
\lim_{p\to\infty}\lambda_1(Z_p^{\wedge k})
&=\lim_{p\to\infty}\lambda_1(p)\cdots\lambda_{k-1}(p)\lambda_k(p) \nonumber\\
&=\lambda_1\dots\lambda_{k-1}\lambda_k \nonumber\\
&>\lambda_1\cdots\lambda_{k-1}\lambda_{k+1}
=\lim_{p\to\infty}\lambda_2(Z_p^{\wedge k}). \label{F-2.9}
\end{align}
Hence it follows that $\lambda_1(Z_p^{\wedge k})$ is a simple eigenvalue of
$Z_p^{\wedge k}$ for every $p$ sufficiently large. Letting $w_{I,J}:=\det W_{I,J}$ for
$I,J\in\cI_d(k)$ we compute
\begin{align*}
(Z_p^{\wedge k})^p
&=(A^{\wedge k})^{p/2}W^{\wedge k}((\diag(b_1,\dots,b_d))^{\wedge k})^p
(W^{\wedge k})^*(A^{\wedge k})^{p/2} \\
&=\diag(a_I^{p/2})_I\bigl[w_{I,J}\bigr]_{I,J}\diag(b_I^p)_I
\bigl[\overline w_{J,I}\bigr]_{I,J}\diag(a_I^{p/2})_I \\
&=\left[\sum_{K\in\cI_d(k)}w_{I,K}\overline w_{J,K}
a_I^{p/2}a_J^{p/2}b_K^p\right]_{I,J} \\
&=\eta_k^p\left[\sum_{K\in\cI_d(k)}w_{I,K}\overline w_{J,K}
\Biggl({a_I^{1/2}a_J^{1/2}b_K\over\eta_k}\Biggr)^p\right]_{I,J},
\end{align*}
where $\eta_k:=\lambda_1\lambda_2\cdots\lambda_k>0$ so that
$$
\eta_k=\max\{a_Ib_K:I,K\in\cI_d(k),\,w_{I,K}\ne0\}
$$
due to Lemma \ref{L-2.3}. We now define
$$
\Delta_k:=\bigl\{(I,K)\in\cI_d(k)^2:w_{I,K}\ne0\ \mbox{and}\ a_Ib_K=\eta_k\bigr\}.
$$
Then we have
\begin{align*}
\biggl({Z_p^{\wedge k}\over\eta_k}\biggr)^p
&=\left[\sum_{K\in\cI_d(k)}w_{I,K}\overline w_{J,K}
\Biggl({a_I^{1/2}a_J^{1/2}b_K\over\eta_k}\Biggr)^p\right]_{I,J} \\
&\longrightarrow
Q:=\left[\sum_{K\in\cI_d(k)}w_{I,K}\overline w_{J,K}\delta_{I,J,K}\right]_{I,J},
\end{align*}
where
$$
\delta_{I,J,K}:=\begin{cases}1 & \text{if $(I,K),(J,K)\in\Delta_k$}, \\
0 & \text{otherwise}.
\end{cases}
$$
Since $Q_{I,I}\ge|w_{I,K}|^2>0$ when $(I,K)\in\Delta_k$, note that $Q\ne0$. Furthermore,
since the eigenvalue $\lambda_1(Z_p^{\wedge k})$ is simple (if $p$ large), it follows
from \eqref{F-2.9} that the limit $Q$ of $\bigl(Z_p^{\wedge k}/\eta_k\bigr)^p$ must be
a rank one projection $\psi\psi^*$ up to a positive scalar multiple, where $\psi$ is a
unit vector in $(\bC^d)^{\wedge k}$. Since the unit eigenvector
$u_1(p)\wedge\dots\wedge u_k(p)$ of $Z_p^{\wedge k}$ corresponding to the largest
(simple) eigenvalue coincides with that of $\bigl(Z_p^{\wedge k}/\eta_k\bigr)^p$, we
conclude that $u_1(p)\wedge\dots\wedge u_k(p)$ converges $\psi$ up to a scalar multiple
$e^{\sqrt{-1}\theta}$. Therefore, by Lemma \ref{L-2.4} the orthogonal projection onto
$\lin\{u_1(p),\dots,u_k(p)\}$ converges as $p\to\infty$.

Assume now that
$$
\lambda_1=\dots=\lambda_{k_1}>\lambda_{k_1+1}=\dots=\lambda_{k_2}
>\dots>\lambda_{k_{s-1}+1}=\dots=\lambda_{k_s}\quad(k_s=d).
$$
From the fact proved above, the orthogonal projection onto
$\lin\{u_1(p),\dots,u_{k_r}(p)\}$ converges for any $r=1,\dots,s-1$, and this is trivial
for $r=s$. Therefore, the orthogonal projection onto
$\lin\{u_{k_{r-1}+1}(p),\dots,u_{k_r}(p)\}$ converges to a projection $P_r$ for any
$r=1,\dots,s$, and thus $Z_p$ converges to $\sum_{r=1}^s\lambda_{k_r}P_r$.
\end{proof}

For $1\le k\le d$ define $\eta_k$ by the right-hand side of \eqref{F-2.8}. Then
Lemma \ref{L-2.3} (see also the proof of Lemma \ref{L-2.1}) implies that, for
$k=1,\dots,d$,
$$
\lambda_k={\eta_k\over\eta_{k-1}}\quad\mbox{if}\quad\eta_k>0
$$
(where $\eta_0:=1$), and $\lambda_k=0$ if $\eta_k=0$. So one can effectively compute the
eigenvalues of $Z:=\lim_{p\to\infty}Z_p$; however, it does not seem that there is a simple
algebraic method to compute the limit matrix $Z$.
\section{The maximal case \label{sec3}}

Let $A$ and $B$ be $d\times d$ positive semidefinite matrices with diagonalizations
\eqref{F-2.1} and \eqref{F-2.2}. For each $d\times d$ matrix $X$ we write
$s_1(X)\ge s_2(X)\ge\dots\ge s_d(X)$ for the singular values of $X$ in decreasing order
with multiplicities. For each $p>0$ and $k=1,\dots,d$, since
$\prod_{i=1}^k\lambda_i(p)=\bigl(\prod_{i=1}^ks_i(A^{p/2}B^{p/2})\bigr)^{2/p}$, by the
majorization results of Gel'fand and Naimark and of Horn (see, e.g., \cite{MOA,Bh1,Hi2}),
we have
$$
\prod_{j=1}^ka_{i_j}b_{n+1-i_j}\le\prod_{j=1}^k\lambda_j(p)\le\prod_{j=1}^ka_jb_j
$$
for any choice of $1\le i_1<i_2<\dots<i_k\le d$, and for $k=d$
$$
\prod_{i=1}^d\lambda_i(p)=\det A\cdot\det B=\prod_{i=1}^da_ib_i.
$$
That is, for any $p>0$,
\begin{equation}\label{F-3.1}
(a_ib_{n+1-i})_{i=1}^d\prec_{(\log)}(\lambda_i(p))_{i=1}^d
\prec_{(\log)}(a_ib_i)_{i=1}^d
\end{equation}
with the notation of log-majorization, see \cite{AH}. Letting $p\to\infty$ gives
\begin{equation}\label{F-3.2}
(a_ib_{n+1-i})_{i=1}^d\prec_{(\log)}(\lambda_i)_{i=1}^d\prec_{(\log)}(a_ib_i)_{i=1}^d
\end{equation}
for the eigenvalues $\lambda_1\ge\dots\ge\lambda_d$ of $Z=\lim_{p\to\infty}Z_p$. In general,
we have nothing to say about the position of $(\lambda_i)_{i=1}^d$ in \eqref{F-3.2}. For
instance, when $V^*W$ becomes the permutation matrix corresponding to a permutation
$(j_1,\dots,j_d)$ of $(1,\dots,d)$, we have $Z_p=V\diag(a_1b_{j_1},\dots,a_db_{j_d})V^*$
independently of $p>0$ so that $(\lambda_i)=(a_ib_{j_i})$.

In this section we clarify the case when $(\lambda_i)_{i=1}^d$ is equal to
$(a_ib_i)_{i=1}^d$, the maximal case in the log-majorization order in \eqref{F-3.2}. To do
this, let $0=i_0<i_1<\cdots<i_{l-1}<i_l=d$ and $0=j_0<j_1<\cdots<j_{m-1}<j_m=d$ be taken
so that
\begin{align*}
&a_1=\dots=a_{i_1}>a_{i_1+1}=\dots=a_{i_2}>\dots>a_{i_{l-1}+1}=\dots=a_{i_l}, \\
&b_1=\dots=b_{j_1}>b_{j_1+1}=\dots=b_{j_2}>\dots>b_{j_{m-1}+1}=\dots=b_{j_m}.
\end{align*}

\begin{thm}\label{T-3.1}
In the above situation the following conditions are equivalent:
\begin{itemize}
\item[(i)] $\lambda_i=a_ib_i$ for all $i=1,\dots,d$;
\item[(ii)] for every $k=1,\dots,d$ so that $i_{r-1}<k\le i_r$ and $j_{s-1}<k\le j_s$,
there are $I_k,J_k\in\cI_d(k)$ such that
$$
\{1,\dots,i_{r-1}\}\subset I_k\subset\{1,\dots,i_r\},\qquad
\{1,\dots,j_{s-1}\}\subset J_k\subset\{1,\dots,j_s\},
$$
$$
\det(V^*W)_{I_k,J_k}\ne0;
$$
\item[(iii)] the property in (ii) holds for every
$k\in\{i_1,\dots,i_{l-1},j_1,\dots,j_{m-1}\}$.
\end{itemize}
\end{thm}

\begin{proof}
(i) $\Leftrightarrow$ (ii).\enspace
By Lemma \ref{L-2.3} condition (ii) means that
$$
\prod_{i=1}^k\lambda_i=\prod_{i=1}^ka_ib_i,\qquad k=1,\dots,d.
$$
It follows (see the proof of Lemma \ref{L-2.1}) that this is equivalent to (i).

(ii) $\Rightarrow$ (iii) is trivial.

(iii) $\Rightarrow$ (i).\enspace
By Lemma \ref{L-2.3} again condition (iii) means that
\begin{equation}\label{F-3.3}
\prod_{i=1}^h\lambda_i=\prod_{i=1}^ha_ib_i\quad
\mbox{for all}\ h\in\{i_1,\dots,i_{l-1},j_1,\dots,j_{m-1}\}.
\end{equation}
This holds also for $h=d$ thanks to \eqref{F-3.2}. We need to prove that
$\prod_{i=1}^k\lambda_i=\prod_{i=1}^ka_ib_i$ for all $k=1,\dots,d$. Now, let
$i_{r-1}<k\le i_r$ and $j_{s-1}<k\le j_s$ as in condition (ii). If $k=i_r$ or $k=j_s$,
then the conclusion has already been stated in \eqref{F-3.3}. So assume that
$i_{r-1}<k<i_r$ and $j_{s-1}<k<j_s$. Set $h_0:=\max\{i_{r-1},j_{s-1}\}$ and
$h_1:=\min\{i_r,j_s\}$ so that $h_0<k<h_1$. By \eqref{F-3.3} for $h=h_0,h_1$ we have
$$
\prod_{i=1}^{h_0}\lambda_i=\prod_{i=1}^{h_0}a_ib_i>0,\qquad
\prod_{i=1}^{h_1}\lambda_i=\prod_{i=1}^{h_1}a_ib_i.
$$
Since $a_i=a_{h_1}$ and $b_i=b_{h_1}$ for $h_0<i\le h_1$, we have
$\prod_{i=h_0+1}^{h_1}\lambda_i=(a_{h_1}b_{h_1})^{h_1-h_0}$.
By \eqref{F-3.2} we furthermore have
$\prod_{i=1}^{h_0+1}\lambda_i\le\prod_{i=1}^{h_0+1}a_ib_i$ and hence
$$
a_{h_1}b_{h_1}\ge\lambda_{h_0+1}\ge\lambda_{h_0+2}\ge\dots\ge\lambda_{h_1}.
$$
Therefore, $\lambda_i=a_{h_1}b_{h_1}$ for all $i$ with $h_0+1<i\le h_1$, from which
$\prod_{i=1}^k\lambda_i=\prod_{i=1}^ka_ib_i$ follows for $h_0<k<h_1$.
\end{proof}

\begin{prop}\label{P-3.2}
Assume that the equivalent conditions of Theorem \ref{T-3.1} hold. Then, for each
$r=1,\dots,l$, the spectral projection of $Z$ corresponding to the set of eigenvalues
$\{a_{i_{r-1}+1}b_{i_{r-1}+1},\dots,a_{i_r}b_{i_r}\}$ is equal to the spectral projection
$\sum_{i=i_{r-1}+1}^{i_r}v_iv_i^*$ of $A$ corresponding to $a_{i_r}$. Hence $Z$ is of the
form
$$
Z=\sum_{i=1}^da_ib_iu_iu_i^*
$$
for some orthonormal set $\{u_1,\dots,u_d\}$ such that
$\sum_{i=i_{r-1}+1}^{i_r}u_iu_i^*=\sum_{i=i_{r-1}+1}^{i_r}v_iv_i^*$ for $r=1,\dots,l$.
\end{prop}

\begin{proof}
In addition to Theorem \ref{T-2.5} we may prove that, for each $k\in\{i_1,\dots,i_{l-1}\}$,
the spectral projection of $Z_p$ corresponding to $\{\lambda_1(p),\dots,\lambda_k(p)\}$
converges to $\sum_{i=1}^kv_iv_i^*$. Assume that $k=i_r$ with $1\le r\le l-1$. When
$j_{s-1}<k<j_s$, by condition (iii) of Theorem \ref{T-3.1} we have
$\det(V^*W)_{\{1,\dots,k\},\{1,\dots,j_{s-1},j_s',\dots,j_k'\}}\ne0$ for some
$\{j_s',\dots,j_k'\}\subset\{j_{s-1}+1,\dots,j_s\}$. By exchanging
$w_{j_s'},\dots,w_{j_k'}$ with $w_{j_{s-1}+1},\dots,w_k$ we may assume that
$\det(V^*W)_{\{1,\dots,k\},\{1,\dots,k\}}\ne0$. Furthermore, by replacing $A$ and $B$ with
$VAV^*$ and $VBV^*$, respectively, we may assume that $V=I$. So we end up assuming that
$$
A=\diag(a_1,\dots,a_d),\qquad B=W\diag(b_1,\dots,b_d)W^*,
$$
and $\det W(1,\dots,k)\ne0$, where $W(1,\dots,k)$ denotes the principal $k\times k$
submatrix of the top-left corner. Let $\{e_1,\dots,e_d\}$ be the standard basis of $\bC^d$.
By Theorem \ref{T-3.1} we have
$$
\lim_{p\to\infty}\lambda_1(Z_p^{\wedge k})
=\prod_{i=1}^ka_ib_i
>\prod_{i=1}^{k-1}a_ib_i\cdot a_{k+1}b_{k+1}
=\lim_{p\to\infty}\lambda_2(Z_p^{\wedge k})
$$
so that the largest eigenvalue of $Z_p^{\wedge k}$ is simple for every sufficiently large
$p$. Let $\{u_1(p),\dots,u_d(p)\}$ be an orthonormal basis of $\bC^d$ for which
$Z_pu_i(p)=\lambda_i(p)u_i(p)$ for $1\le i\le d$. Then $u_1(p)\wedge\dots\wedge u_k(p)$
is the unit eigenvector of $Z_p^{\wedge k}$ corresponding to the eigenvalue
$\lambda_1(Z_p^{\wedge k})$. We now show that $u_1(p)\wedge\dots\wedge u_k(p)$ converges
to $e_1\wedge\dots\wedge e_k$ in $(\bC^d)^{\wedge k}$. We observe that
$$
(A^{\wedge k})^{p/2}=\diag\bigl(a_I^{p/2}\bigr)_I
=a_{\{1,\dots,k\}}^{p/2}\diag\biggl(1,\alpha_2^{p/2},\dots,
\alpha_{d\choose k}^{p/2}\biggr)
$$
with respect to the basis $\bigl\{e_{i_1}\wedge\dots\wedge e_{i_k}:
I=\{i_1,\dots,i_k\}\in\cI_d(k)\bigr\}$, where the first diagonal entry $1$
corresponds to $e_1\wedge\dots\wedge e_k$ and $0\le\alpha_h<1$ for $2\le h\le{d\choose k}$.
Similarly,
$$
\bigl((\diag(b_1,\dots,b_d))^{\wedge k}\bigr)^p
=b_{\{1,\dots,k\}}^p\diag\biggl(1,\beta_2^p,\dots,\beta_{d\choose k}^p\biggr),
$$
where $0\le\beta_h\le1$ for $2\le h\le{d\choose k}$. Moreover, $W^{\wedge k}$ is given as
$$
W^{\wedge k}=\bigl[w_{I,J}\bigr]_{I,J}
=\begin{bmatrix}w_{11}&\cdots&w_{1{d\choose k}}\\
\vdots&\ddots&\vdots\\
w_{{d\choose k}1}&\cdots&w_{{d\choose k}{d\choose k}}
\end{bmatrix},
$$
where $w_{I,J}=\det W_{I,J}$ and so $w_{11}=\det W(1,\dots,k)\ne0$. As in the proof of
Theorem \ref{T-2.5} we now compute
\begin{align*}
(Z_p^{\wedge k})^p
&=(A^{\wedge k})^{p/2}W^{\wedge k}\bigl((\diag(b_1,\dots,b_d))^{\wedge k}\bigr)^p
(W^{\wedge k})^*(A^{\wedge k})^{p/2} \\
&=\bigl(a_{\{1,\dots,k\}}b_{\{1,\dots,k\}}\bigr)^p
\left[\sum_{h=1}^{d\choose k}w_{ih}\overline w_{jh}
\alpha_i^{p/2}\alpha_j^{p/2}\beta_h^p\right]_{i,j=1}^{d\choose k},
\end{align*}
where $\alpha_1=\beta_1=1$. As $p\to\infty$ we have
$$
\left[\sum_{h=1}^{d\choose k}w_{ih}\overline w_{jh}
\alpha_i^{p/2}\alpha_j^{p/2}\beta_h^p\right]
\longrightarrow\diag\Biggl(\sum_{h:\beta_h=1}|w_{1h}|^2,0,\dots,0\Biggr)
$$
Since the unit eigenvector of $Z_p^{\wedge k}$ corresponding to the largest eigenvalue
coincides with that of $\biggl[\sum_{h=1}^{d\choose k}w_{ih}\overline w_{jh}
\alpha_i^{p/2}\alpha_j^{p/2}\beta_h^p\biggr]$, it follows that
$u_1(p)\wedge\dots\wedge u_k(p)$ converges to $e_1\wedge\dots\wedge e_k$ up to a scalar
multiple $e^{\sqrt{-1}\theta}$, $\theta\in\bR$. By Lemma \ref{L-2.4} this implies the
desired assertion.
\end{proof}

\begin{cor}\label{C-3.3}
If the eigenvalues $a_1,\dots,a_d$ of $A$ are all distinct and the conditions of Theorem \ref{T-3.1} hold, then
$$
\lim_{p\to\infty}(A^{p/2}B^pA^{p/2})^{1/p}=V\diag(a_1b_1,a_2b_2,\dots,a_db_d)V^*.
$$
\end{cor}

In particular, when the eigenvalues of $A$ are all distinct and so are those of $B$, the conditions
of Theorem \ref{T-3.1} means that all the leading principal minors of $V^*W$ are non-zero.

\section{Extension to more than two matrices \label{sec4}}

Let $A_1,\dots,A_m$ be $d\times d$ positive semidefinite matrices with diagonalizations
$$
A_l=V_lD_lV_l^*,\qquad D_l=\diag\bigl(a_1^{(l)},\dots,a_d^{(l)}\bigr),
\qquad1\le l\le m.
$$
For each $p>0$ consider the positive semidefinite matrix
\begin{align*}
Z_p&:=\bigl(A_1^{p/2}A_2^{p/2}\cdots A_{m-1}^{p/2}A_m^pA_{m-1}^{p/2}\cdots
A_1^{p/2}A_1^{p/2}\bigr)^{1/p}, \\
&\ =V_1\bigl(D_1^{p/2}W_1\cdots D_{m-1}^{p/2}W_{m-1}D_m^pW_{m-1}^*D_{m-1}^{p/2}\cdots
W_1^*D_1^{p/2}\bigr)^{1/p}V_1^*,
\end{align*}
where
$$
W_l:=V_l^*V_{l+1}=\Bigl[w_{ij}^{(l)}\Bigr]_{i,j=1}^d,\qquad1\le l\le m-1.
$$
The eigenvalues of $Z_p$ are denoted as
$\lambda_1(p)\ge\lambda_2(p)\ge\dots\ge\lambda_d(p)$ in decreasing order.
Although the log-majorization in \eqref{F-2.5} is no longer available in the present
situation, we can extend Lemma \ref{L-2.2} as follows.

\begin{lemma}\label{L-4.1}
The limit $\lambda_1:=\lim_{p\to\infty}\lambda_1(p)$ exists and
\begin{equation}\label{F-4.1}
\lambda_1=\max\bigl\{a_{i_1}^{(1)}a_{i_2}^{(2)}\cdots a_{i_m}^{(m)}:
\bw(i_1,i_2,\dots,i_m)\ne0\bigr\},
\end{equation}
where
\begin{align*}
&\bw(i_1,i_2,\dots,i_m) \\
&:=\sum\Bigl\{w_{i_1j_2}^{(1)}w_{j_2j_3}^{(2)}\cdots w_{j_{m-1}i_m}^{(m)}:
1\le j_2,\dots,j_{m-1}\le d,\,a_{j_2}^{(2)}\cdots
a_{j_{m-1}}^{(m-1)}=a_{i_2}^{(2)}\cdots a_{i_{m-1}}^{(m-1)}\Bigr\}.
\end{align*}
Moreover, $a_1^{(1)}\cdots a_1^{(m)}\ge\lambda_1\ge a_d^{(1)}\cdots a_d^{(m)}$.
\end{lemma}

\begin{proof}
We notice that
\begin{align*}
\bigl[V_1^*Z_p^pV_1\bigr]_{ii}
&=\bigl[D_1^{p/2}W_1\cdots D_{m-1}^{p/2}W_{m-1}D_m^pW_{m-1}^*D_{m-1}^{p/2}
\cdots W_1^*D_1^{p/2}\bigr]_{ii} \\
&=\sum_{i_2,\dots,i_{m-1},k,j_{m-1},\dots,j_2}
\bigl(a_i^{(1)}\bigr)^{p/2}w_{ii_2}^{(1)}\bigl(a_{i_2}^{(2)}\bigr)^{p/2}\cdots
w_{i_{m-2}i_{m-1}}^{(m-2)}\bigl(a_{i_{m-1}}^{(m-1)}\bigr)^{p/2} \\
&\qquad\times
w_{i_{m-1}k}^{(m-1)}\bigl(a_k^{(m)}\bigr)^p
\overline w_{j_{m-1}k}^{(m-1)}\bigl(a_{j_{m-1}}^{(m-1)}\bigr)^{p/2}
\overline w_{j_{m-2}j_{m-1}}^{(m-2)}\cdots\bigl(a_{j_2}^{(2)}\bigr)^{p/2}
\overline w_{ij_2}^{(1)}\bigl(a_i^{(1)}\bigr)^{p/2} \\
&=\sum_k\sum_{i_2,\dots,i_{m-1}}
w_{ii_2}^{(1)}w_{i_2i_3}^{(2)}\cdots w_{i_{m-1}k}^{(m-1)}
\bigl(a_i^{(1)}a_{i_2}^{(2)}\cdots a_{i_{m-1}}^{(m-1)}a_k^{(m)}\bigr)^{p/2} \\
&\qquad\times
\sum_{j_2,\dots,j_{m-1}}
\overline{w_{ij_2}^{(1)}w_{j_2j_3}^{(2)}\cdots w_{j_{m-1}k}^{(m-1)}}
\bigl(a_i^{(1)}a_{j_2}^{(2)}\cdots a_{j_{m-1}}^{(m-1)}a_k^{(m)}\bigr)^{p/2} \\
&=\sum_k\Bigg|\sum_{j_2,\dots,j_{m-1}}
w_{ij_2}^{(1)}w_{j_2j_3}^{(2)}\cdots w_{j_{m-1}k}^{(m-1)}
\bigl(a_i^{(1)}a_{j_2}^{(2)}\cdots a_{j_{m-1}}^{(m-1)}a_k^{(m)}\bigr)^{p/2}
\Bigg|^2.
\end{align*}
Let $\eta$ be the right-hand side of \eqref{F-4.1}. From the above expression we have
\begin{align*}
\lambda_1(p)^p&\le\Tr V_1^*Z_p^pV_1 \\
&=\sum_{i,k}\Bigg|\sum_{j_2,\dots,j_{m-1}}
w_{ij_2}^{(1)}w_{j_2j_3}^{(2)}\cdots w_{j_{m-1}k}^{(m-1)}
\bigl(a_i^{(1)}a_{j_2}^{(2)}\cdots a_{j_{m-1}}^{(m-1)}a_k^{(m)}\bigr)^{p/2}
\Bigg|^2 \\
&\le M\eta^p,
\end{align*}
where $M>0$ is a constant independent of $p$. Therefore,
$\limsup_{p\to\infty}\lambda_1(p)\le\eta$. On the other hand, let
$(i,i_2,\dots,i_{m-1},k)$ be such that $a_i^{(1)}a_{i_2}^{(2)}\cdots a_{i_{m-1}}^{(m-1)}a_k^{(m)}=\eta$,
and let $\delta:=|\bw(i,i_2,\dots,i_{m-1},k)|>0$. Then we have
$$
\Bigg|\sum_{j_2,\dots,j_{m-1}}
w_{ij_2}^{(1)}w_{j_2j_3}^{(2)}\cdots w_{j_{m-1}k}^{(m-1)}
\bigl(a_i^{(1)}a_{j_2}^{(2)}\cdots a_{j_{m-1}}^{(m-1)}a_k^{(m)}\bigr)^{p/2}\Bigg|
\ge\delta\eta^{p/2}-M'\alpha^{p/2}
$$
for some constants $M'>0$ and $\alpha>0$ with $\alpha<\eta$. Therefore, for sufficiently
large $p$ we have $\delta\eta^{p/2}-M'\alpha^{p/2}>0$ and
$$
d\lambda_1(p)^p\ge\Tr V_1^*Z_p^pV_1
\ge\bigl(\delta\eta^{p/2}-M'\alpha^{p/2}\bigr)^2
=\delta^2\eta^p\biggl(1-{M'\over\delta}\biggl({\alpha\over\eta}\biggr)^{p/2}\biggr)^2
$$
so that $\liminf_{p\to\infty}\lambda_1(p)\ge\eta$. The latter assertion is obvious.
\end{proof}

\begin{lemma}\label{L-4.2}
For every $i=1,\dots,d$ the limit $\lambda_i:=\lim_{p\to\infty}\lambda_i(p)$ exists.
\end{lemma}

\begin{proof}
For every $k=1,\dots,d$ apply Lemma \ref{L-4.1} to $A_1^{\wedge k},\dots,A_m^{\wedge k}$
to see that
$$
\lim_{p\to\infty}\lambda_1(p)\lambda_2(p)\cdots\lambda_k(p)
$$
exists. Hence, the limit $\lim_{p\to\infty}\lambda_i(p)$ exists for $i=1,\dots,d$ as in
the proof of Lemma \ref{L-2.1}.
\end{proof}

\begin{thm}\label{T-4.3}
For every $d\times d$ positive semidefinite matrices $A_1,\dots,A_m$ the matrix
$$
Z_p=\bigl(A_1^{p/2}A_2^{p/2}\cdots A_{m-1}^{p/2}A_m^pA_{m-1}^{p/2}\cdots
A_2^{p/2}A_1^{p/2}\bigr)^{1/p}
$$
converges as $p\to\infty$.
\end{thm}

\begin{proof}
The proof is similar to that of Theorem \ref{T-2.5}. Choose an orthogonal basis
$\{u_1(p),\dots,u_d(p)\}$ of $\bC^d$ such that $Z_pu_i(p)=\lambda_i(p)u_i(p)$ for
$1\le i\le d$. Let $k$ ($1\le k<d$) be such that
$\lambda_1\ge\dots\ge\lambda_k>\lambda_{k+1}$. Since \eqref{F-2.9} holds in the present
case too, $\lambda_1(Z_p^{\wedge k})$ is a simple eigenvalue of $Z_p^{\wedge k}$ for
every $p$ sufficiently large. For $I,J\in\cI_d(k)$ we write
$w_{I,J}^{(l)}:=\det W_{I,J}^{(l)}$ for $1\le l\le m-1$ and
$a_I^{(l)}:=\prod_{i\in I}a_i^{(l)}$ for $1\le l\le m$. We have
\begin{align*}
&\bigl[V_1^{*\wedge k}(Z_p^{\wedge k})^pV_1^{\wedge k}\bigr]_{I,J} \\
&\qquad=\sum_{K\in\cI_d(k)}\sum_{I_2,\dots,I_{m-1}}
w_{I,I_2}^{(1)}w_{I_2,I_3}^{(2)}\cdots w_{I_{m-1},K}^{(m-1)}
\bigl(a_I^{(1)}a_{I_2}^{(2)}\cdots a_{I_{m-1}}^{(m-1)}a_K^{(m)}\bigr)^{p/2} \\
&\qquad\qquad\times\sum_{J_2,\dots,J_{m-1}}
\overline{w_{J,J_2}^{(1)}w_{J_2,J_3}^{(2)}\cdots w_{J_{m-1},K}^{(m-1)}}
\bigl(a_J^{(1)}a_{J_2}^{(2)}\cdots a_{J_{m-1}}^{(m-1)}a_K^{(m)}\bigr)^{p/2} \\
&\qquad=\eta_k^p\sum_{K\in\cI_d(k)}\sum_{I_2,\dots,I_{m-1}}
w_{I,I_2}^{(1)}w_{I_2,I_3}^{(2)}\cdots w_{I_{m-1},K}^{(m-1)}
\Biggl({a_I^{(1)}a_{I_2}^{(2)}\cdots a_{I_{m-1}}^{(m-1)}a_K^{(m)}\over\eta_k}
\Biggr)^{p/2} \\
&\qquad\qquad\quad\times\sum_{J_2,\dots,J_{m-1}}
\overline{w_{J,J_2}^{(1)}w_{J_2,J_3}^{(2)}\cdots w_{J_{m-1},K}^{(m-1)}}
\Biggl({a_J^{(1)}a_{J_2}^{(2)}\cdots a_{J_{m-1}}^{(m-1)}a_K^{(m)}\over\eta_k}
\Biggr)^{p/2},
\end{align*}
where
$$
\eta_k:=\lambda_1\lambda_2\cdots\lambda_k
=\max\bigl\{a_{I_1}^{(1)}a_{I_2}^{(2)}\cdots a_{I_{m-1}}^{(m-1)}a_{I_m}^{(m)}:
\bw_k(I_1,I_2,\dots,I_{m-1},I_m)\ne0\bigr\}
$$
and
\begin{align*}
&\bw_k(I_1,I_2,\dots,I_{m-1},I_m) \\
&:=\sum\Bigl\{w_{I_1J_2}^{(1)}w_{J_2J_3}^{(2)}\cdots w_{J_{m-1}I_m}^{(m-1)}:
J_2,\dots,J_{m-1}\in\cI_d(k),\,a_{J_2}^{(2)}\cdots
a_{J_{m-1}}^{(m-1)}=a_{I_2}^{(2)}\cdots a_{I_{m-1}}^{(m-1)}\Bigr\}.
\end{align*}
We see that
$$
V_1^{*\wedge k}\biggl({Z_p^{\wedge k}\over\eta_k}\biggr)^pV_1^{\wedge k}
\longrightarrow
Q:=\Biggl[\sum_{K\in\cI_d(k)}\bv_k(I,K)\overline{\bv_k(J,K)}\Biggr]_{I,J}
\quad\mbox{as}\ p\to\infty,
$$
where
$$
\bv_k(I,K):=\bw_k(I,I_2,\dots,I_{m-1},K)
$$
if $\bw_k(I,I_2,\dots,I_{m-1},K)\ne0$ and
$a_I^{(1)}a_{I_2}^{(2)}\cdots a_{I_{m-1}}^{(m-1)}a_K^{(m)}=\eta_k$ for some
$I_2,\dots,I_{m-1}\in\cI_d(k)$, and otherwise $\bv_k(I,K):=0$. Since
$Q_{I,I}\ge|\bv_k(I,K)|^2>0$ for some $I,K\in\cI_d(k)$, note that $Q\ne0$. The remaining
proof is the same as in that of Theorem \ref{T-2.5}.
\end{proof}
\section{Limit of $(A^p\,\#\,B^p)^{1/p}$ as $p\to\infty$ \label{sec5}}

Another problem, seemingly more interesting, is to know what is shown on the convergence
$(A^p\,\sigma\,B^p)^{1/p}$ as $p\to\infty$, the anti-version of \eqref{F-1.2} (or Theorem
\ref{T-B.1}). For example, when $\sigma=\triangledown$, the arithmetic mean, the increasing
limit of $(A^p\,\triangledown\,B^p)^{1/p}=\bigl((A^p+B^p)/2\bigr)^{1/p}$ as $p\to\infty$
exists and
\begin{equation}\label{F-5.1}
A\vee B:=\lim_{p\to\infty}(A^{-p}\,\triangledown\,B^{-p})^{-1/p}
=\lim_{p\to\infty}(A^p+B^p)^{1/p}
\end{equation}
is the supremum of $A,B$ with respect to some spectral order among Hermitian matrices,
see \cite{Ka79} and \cite[Lemma 6.5]{An}. When $\sigma=\,!$, the harmonic mean, we have the
infimum counterpart $A\wedge B:=\lim_{p\to\infty}(A^p\,!\,B^p)^{1/p}$, the
decreasing limit.

In this section we are interested in the case where $\sigma=\#$, the geometric mean. For
each $p>0$ and $d\times d$ positive semidefinite matrices $A,B$ with the diagonalizations
in \eqref{F-2.1} and \eqref{F-2.2} we define
\begin{equation}\label{F-5.2}
G_p:=(A^p\,\#\,B^p)^{2/p},
\end{equation}
which is given as $\bigl(A^{p/2}(A^{-p/2}B^pA^{-p/2})^{1/2}A^{p/2}\bigr)^{2/p}$ if $A>0$.
The eigenvalues of $G_p$ are denoted as $\lambda_1(G_p)\ge\dots\ge\lambda_d(G_p)$ in
decreasing order.

\begin{prop}\label{P-5.1}
For every $i=1,\dots,d$  the limit
$$
\widehat\lambda_i:=\lim_{p\to\infty}\lambda_i(G_p)
$$
exists, and $a_1b_1\ge\widehat\lambda_1\ge\dots\ge\widehat\lambda_d\ge a_db_d$. Furthermore,
\begin{equation}\label{F-5.3}
(a_ib_{d+1-i})_{i=1}^d\prec_{(\log)}
\bigl(\widehat\lambda_i\bigr)_{i=1}^d\prec_{(\log)}(a_ib_i)_{i=1}^d.
\end{equation}
\end{prop}

\begin{proof}
Since $(a_1b_1)^{p/2}I\ge A^p\,\#\,B^p\ge(a_db_d)^{p/2}I$, we have
$a_1b_1\ge\lambda_i(G_p)\ge a_db_d$ for all $i=1,\dots,d$ and $p>0$. By the log-majorization
result in \cite[Theorem 2.1]{AH}, for every $k=1,\dots,d$ we have
\begin{equation}\label{F-5.4}
\prod_{i=1}^k\lambda_i(G_p)\ge\prod_{i=1}^k\lambda_i(G_q)
\quad\mbox{if\quad $0<p<q$}.
\end{equation}
This implies that the limit of $\prod_{i=1}^k\lambda_i(G_p)$ as $p\to\infty$ exists for
every $k=1,\dots,d$, and hence the limit $\lambda_i(G_p)$ exists for $i=1,\dots,d$ as in
the proof of Lemma \ref{L-2.1}.

To prove the latter assertion, it suffices to show that
\begin{equation}\label{F-5.5}
(a_ib_{d+1-i})_{i=1}^d\prec_{(\log)}
(\lambda_i(G_1))_{i=1}^d\prec_{(\log)}(a_ib_i)_{i=1}^d
\end{equation}
for $G_1=(A\,\#\,B)^2$. Indeed, applying this to $A^p$ and $B^p$ we have
$$
(a_ib_{d+1-i})_{i=1}^d\prec_{(\log)}
(\lambda_i(G_p))_{i=1}^d\prec_{(\log)}(a_ib_i)_{i=1}^d
$$
so that \eqref{F-5.3} follows by letting $p\to\infty$. To prove \eqref{F-5.5}, we may by
continuity assume that $A>0$. By \cite[Corollary 2.3]{AH} and \eqref{F-3.1} we have
$$
(\lambda_i(G_1))_{i=1}^d\prec_{(\log)}
\bigl(\lambda_i(A^{1/2}BA^{1/2})\bigr)_{i=1}^d\prec_{(\log)}(a_ib_i)_{i=1}^d.
$$
Since $G_1^{1/2}A^{-1}G_1^{1/2}=B$, there exists a unitary matrix $V$ such that
$A^{-1/2}G_1A^{-1/2}=VBV^*$ and hence $G_1=A^{1/2}VBV^*A^{1/2}$. Since
$\lambda_i(VBV^*)=b_i$, by the majorization of Gel'fand and Naimark we have
$$
(a_ib_{d+1-i})_{i=1}^d\prec_{(\log)}(\lambda_i(G_1))_{i=1}^d,
$$
proving \eqref{F-5.5}
\end{proof}

In view of \eqref{F-2.5} and \eqref{F-5.4} we may consider $G_p$ as the complementary
counterpart of $Z_p$ in some sense; yet it is also worth noting
that $G_p$ is symmetric in $A$ and $B$ while $Z_p$ is not. Our ultimate goal is to
prove the existence of the limit of $G_p$ in \eqref{F-5.2} as $p\to\infty$ similarly to
Theorem \ref{T-2.5} and to clarify, similarly to Theorem \ref{T-3.1}, the minimal case
when $\bigl(\widehat\lambda_i\bigr)_{i=1}^d$ is equal to the decreasing rearrangement of
$(a_ib_{d+1-i})_{i=1}^d$. However, the problem seems much more difficult, and we can currently
settle the special case of $2\times2$ matrices only.

\begin{prop}\label{P-5.2}
Let $A$ and $B$ be $2\times2$ positive semidefinite matrices with the diagonalizations
\eqref{F-2.1} and \eqref{F-2.2} with $d=2$. Then $G_p$ in \eqref{F-5.2} converges as
$p\to\infty$ to a positive semidefinite matrix whose eigenvalues are
$$
\bigl(\widehat\lambda_1,\widehat\lambda_2\bigr)=\begin{cases}
(a_1b_1,a_2b_2) & \text{if $(V^*W)_{12}=0$}, \\
(\max\{a_1b_2,a_2b_1\},\min\{a_1b_2,a_2b_1\}) & \text{if $(V^*W)_{12}\ne0$}.
\end{cases}
$$
\end{prop}

\begin{proof}
Since
$$
G_p=V\bigl((\diag(a_1,a_2))^p\,\#\,(V^*W\diag(b_1,b_2)V^*W)^p\bigr)^{2/p}V^*,
$$
we may assume without loss of generality that $V=I$ (then $V^*W=W$).

First, when $W_{12}=0$ (hence $W$ is diagonal), we have for every $p>0$
$$
G_p=\diag(a_1b_1,a_2b_2).
$$
Next, when $W_{11}=0$ (hence $W=\begin{bmatrix}0&w_1\\w_2&0\end{bmatrix}$ with
$|w_1|=|w_2|=1$), we have for every $p>0$
$$
G_p=\diag(a_1b_2,a_2b_1).
$$

In the rest it suffices to consider the case where
$W=\begin{bmatrix}w_{11}&w_{12}\\w_{21}&w_{22}\end{bmatrix}$ with $w_{ij}\ne0$ for all
$i,j=1,2$. First, assume that $\det A=\det B=1$ so that $a_1a_2=b_1b_2=1$. For every $p>0$,
since $\det A^p=\det B^p=1$, it is known \cite[Proposition 3.11]{Mo} (also
\cite[Proposition 4.1.12]{Bh2}) that
$$
A^p\,\#\,B^p={A^p+B^p\over\sqrt{\det(A^p+B^p)}}
$$
so that
$$
G_p={(A^p+B^p)^{2/p}\over\bigl(\det(A^p+B^p)\bigr)^{1/p}}.
$$
Compute
\begin{equation}\label{F-5.6}
A^p+B^p=\begin{bmatrix}
a_1^p+|w_{11}|^2b_1^p+|w_{12}|^2b_2^p&
w_{11}\overline{w_{21}}b_1^p+w_{12}\overline{w_{22}}b_2^p\\
\overline{w_{11}}w_{21}b_1^p+\overline{w_{12}}w_{22}b_2^p&
a_2^p+|w_{21}|^2b_1^p+|w_{22}|^2b_2^p
\end{bmatrix}
\end{equation}
and
\begin{align}
\det(A^p+B^p)
&=1+|w_{21}|^2(a_1b_1)^p+|w_{22}|^2(a_1b_2)^p+|w_{11}|^2(a_2b_1)^p+|w_{12}|^2(a_2b_2)^p
\nonumber\\
&\qquad+|w_{11}w_{22}-w_{12}w_{21}|^2. \label{F-5.7}
\end{align}
Hence we have
$$
\lim_{p\to\infty}\bigl(\det(A^p+B^p)\bigr)^{1/p}=a_1b_1,\qquad
\lim_{p\to\infty}\bigl(\Tr(A^p+B^p)\bigr)^{1/p}=\max\{a_1,b_1\}.
$$
Therefore, thanks to \eqref{F-5.1} we have
$$
\lim_{p\to\infty}G_p={(A\vee B)^2\over a_1b_1}.
$$
Since
$$
{1\over2}\,\Tr(A^p\,\#\,B^p)\le\bigl(\lambda_1(G_p)\bigr)^{p/2}
\le\Tr(A^p\,\#\,B^p),
$$
we obtain
\begin{align*}
\widehat\lambda_1&=\lim_{p\to\infty}\bigl(\Tr(A^p\,\#\,B^p)\bigr)^{2/p}
=\lim_{p\to\infty}{\bigl(\Tr(A^p+B^p)\bigr)^{2/p}\over\bigl(\det(A^p+B^p)\bigr)^{1/p}} \\
&={\max\{a_1^2,b_1^2\}\over a_1b_1}
=\max\biggl\{{a_1\over b_1},{b_1\over a_1}\biggr\}
=\max\{a_1b_2,a_2b_1\}.
\end{align*}
Furthermore, $\widehat\lambda_2=\min\{a_1b_2,a_2b_1\}$ follows since
$\widehat\lambda_1\widehat\lambda_2=1$.

For general $A,B>0$ let $\alpha:=\sqrt{\det A}$ and $\beta:=\sqrt{\det B}$. Since
$$
G_p=\alpha\beta\bigl((\alpha^{-1}A)^p\,\#\,(\beta^{-1}B)^p\bigr)^{2/p},
$$
we see from the above case that $G_p$ converges as $p\to\infty$ and
$$
\widehat\lambda_1
=\alpha\beta\max\{(\alpha^{-1}a_1)(\beta^{-1}b_2),(\alpha^{-1}a_2)(\beta^{-1}b_1)\}
=\max\{a_1b_2,a_2b_1\},
$$
and similarly for $\widehat\lambda_2$.

The remaining is the case when $a_2$ and/or $b_2=0$. We may assume that $a_1,b_1>0$ since
the case $A=0$ or $B=0$ is trivial. When $a_2=b_2=0$, since $a_1^{-1}A$ and $b_1^{-1}B$
are non-commuting rank one projections, we have $G_p=0$ for all $p>0$ by
\cite[(3.11)]{KA}. Finally, assume that $a_2=0$ and $B>0$. Then we may assume that
$a_1=1$ and $\det B=1$. For $\eps>0$ set $A_\eps:=\diag(1,\eps^2)$. Since
$\det(\eps^{-1}A_\eps)=1$, we have
$$
A_\eps^p\,\#\,B^p=\eps^{p/2}\bigl((\eps^{-1}A_\eps)^p\,\#\,B^p\bigr)
=\eps^{p/2}\,{(\eps^{-1}A_\eps)^p+B^p\over
\sqrt{\det\bigl((\eps^{-1}A_\eps)^p+B^p\bigr)}}.
$$
By use of \eqref{F-5.6} and \eqref{F-5.7} with $a_1=\eps^{-1}$ and $a_2=\eps$ we compute
$$
A^p\,\#\,B^p=\lim_{\eps\searrow0}A_\eps^p\,\#\,B^p
=\bigl(|w_{21}|^2b_1^p+|w_{22}|^2b_2^p\bigr)^{-1/2}\,\diag(1,0)
$$
so that
$$
\lim_{p\to\infty}G_p=\diag(b_1^{-1},0)=\diag(b_2,0),
$$
which is the desired assertion in this final situation.
\end{proof}

\appendix

\section{Proof of Lemma \ref{L-2.4}}

We may assume that $\cH=\bC^d$ by fixing an orthonormal basis of $\cH$. Let $G(k,d)$
denote the Grassmannian manifold consisting of $k$-dimensional subspaces of $\cH$. Let
$\cO_{k,d}$ denote the set of all $u=(u_1,\dots,u_k)\in\cH^k$ such that $u_1,\dots,u_k$
are orthonormal in $\cH$. Consider $\cO_{k,d}$ as a metric space with the metric
$$
d_2(u,v):=\Biggl(\sum_{i=1}^k\|u_i-v_i\|^2\Biggr)^{1/2},\qquad
u=(u_1,\dots,u_k),\ v=(v_1,\dots,v_k)\in\cH^k.
$$
Moreover, let $\widetilde\cH_{k,d}$ be the set of projectivised vectors
$u=u_1\wedge\dots\wedge u_k$ in $\cH^{\wedge k}$ of norm $1$, i.e., the quotient
space of $\cH_{k,d}:=\{u\in\cH^{\wedge k}:u=u_1\wedge\dots\wedge u_k,\,\|u\|=1\}$ under
the equivalent relation $u\sim v$ on $\cH_{k,d}$ defined as $u=e^{i\theta}v$ for some
$\theta\in\bR$. We then have the commutative diagram:
$$
\setlength{\unitlength}{1mm}
\begin{picture}(50,27)(0,0)
\put(0,22){$\cO_{k,d}$}
\put(10,23){\vector(1,0){15}}
\put(16,25){$\pi$}
\put(28,22){$G(k,d)$}
\put(28,0){$\widetilde\cH_{k,d}$}
\put(10,20){\vector(1,-1){15}}
\put(13,9){$\widetilde\pi$}
\put(32,19){\vector(0,-1){12}}
\put(34,12){$\phi$}
\end{picture}
$$
where $\pi$ and $\widetilde\pi$ are surjective maps defined for
$u=(u_1,\dots,u_k)\in\cO_{k,d}$ as
\begin{align*}
\pi(u)&:=\lin\{u_1,\dots,u_k\}, \\
\widetilde\pi(u)&:=[u_1\wedge\dots\wedge u_k],\ \mbox{the equivalence class of
$u_1\wedge\dots\wedge u_k$},
\end{align*}
and $\phi$ is the canonical representation of $G(k,d)$ by the $k$th antisymmetric tensors
(or the $k$th exterior products).

As shown in \cite{FGP}, the standard Grassmannian topology on $G(k,d)$ is the final
topology (the quotient topology) from the map $\pi$ and it coincides with the topology
induced by the gap metric:
$$
d_\gap(\cU,\cV):=\|P_\cU-P_\cV\|
$$
for $k$-dimensional subspaces $\cU,\cV$ of $\cH$ and the orthogonal projections
$P_\cU,P_\cV$ onto them. On the other hand, consider the quotient topology on
$\widetilde\cH_{k,d}$ induced from the norm on $\cH_{k,d}\subset\cH^{\wedge k}$, which is
determined by the metric
$$
\widetilde d(\widetilde\pi(u),\widetilde\pi(v)):=\inf_{\theta\in\bR}
\|u_1\wedge\cdots\wedge u_k-e^{\sqrt{-1}\theta}v_1\wedge\cdots\wedge v_k\|,
\qquad u,v\in\cO_{k,d}.
$$
It is easy to prove that
$\widetilde\pi:(\cO_{k,d},d_2)\to(\widetilde\cH_{k,d},\widetilde d)$ is continuous. Since
$(\cO_{k,d},d_2)$ is compact, it thus follows that the final topology on
$\widetilde\cH_{k,d}$ from the map $\widetilde\pi$ coincides with the
$\widetilde d$-topology.

It is clear from the above commutative diagram that the final topology on $G(k,d)$ from
$\pi$ is homeomorphic via $\phi$ to that on $\widetilde\cH_{k,d}$ from $\widetilde\pi$.
Hence $\phi$ is a homeomorphism from $(G(k,d),d_\gap)$ onto
$(\widetilde\cH_{k,d},\widetilde d)$. From the homogeneity of $(G(k,d),d_\gap)$ and
$(\widetilde\cH_{k,d},\widetilde d)$ under the unitary transformations there exist
constant $\alpha,\beta>0$ (depending on only $k,d$) such that
$$
\alpha\|P_{\pi(u)}-P_{\pi(v)}\|\le\widetilde d(\widetilde\pi(u),\widetilde\pi(v))
\le\beta\|P_{\pi(u)}-P_{\pi(v)}\|,\qquad u,v\in\cO_{k,d},
$$
which is the desired inequality.

\section{Proof of \eqref{F-1.2}}

This appendix is aimed to supply the proof of \eqref{F-1.2} for matrices $A,B\ge0$. Throughout
the appendix let $A,B$ be $d\times d$ positive semidefinite matrices with the support
projections $A^0,B^0$. We define $\log A$ in the generalized sense as
$$
\log A:=(\log A)A^0,
$$
i.e., $\log A$ is defined by the usual functional calculus on the range of $A^0$ and it
is zero on the range of $A^{0\perp}=I-A^0$, and similarly $\log B:=(\log B)B^0$. We write
$P_0:=A^0\wedge B^0$ and
$$
\log A\,\dot+\log B:=P_0(\log A)P_0+P_0(\log B)P_0.
$$
Note \cite[Section 4]{HP} that
\begin{align}
P_0\exp(\log A\,\dot+\log B)
&=\lim_{\eps\searrow0}\exp(\log(A+\eps A^{0\perp})+\log(B+\eps B^{0\perp})) \nonumber\\
&=\lim_{\eps\searrow0}\exp(\log(A+\eps I)+\log(B+\eps I)). \label{F-B.1}
\end{align}

Now, let $\sigma$ be an operator mean with the representing operator monotone function $f$ on
$(0,\infty)$, and let $\alpha:=f'(1)$. Note that $0\le\alpha\le1$ and if $\alpha=0$ (resp.,
$\alpha=1$) then $A\,\sigma\,B=A$ (resp., $A\,\sigma B=B)$ so that $(A^p\,\sigma\,B^p)^{1/p}=A$
(resp., $(A^p\,\sigma\,B^p)^{1/p}=B$) for all $A,B\ge0$ and $p>0$. So in the rest we assume
that $0<\alpha<1$.

\begin{thm}\label{T-B.1}
With the above assumptions, for every $A,B\ge0$,
\begin{equation}\label{F-B.2}
\lim_{p\searrow0}(A^p\,\sigma\,B^p)^{1/p}
=P_0\exp((1-\alpha)\log A\dot+\alpha\log B).
\end{equation}
\end{thm}

From \eqref{F-B.1} we may write
\begin{align*}
\lim_{p\searrow0}(A^p\,\sigma\,B^p)^{1/p}
&=\lim_{\eps\searrow0}\exp((1-\alpha)\log(A+\eps I)+\alpha\log(B+\eps I)) \\
&=\lim_{\eps\searrow0}\lim_{p\searrow0}((A+\eps I)^p\,\sigma\,(B+\eps I)^p)^{1/p}.
\end{align*}

The next lemma is essential to prove the theorem. The proof of the lemma is a slight
modification of that of \cite[Lemma 4.1]{HP}.

\begin{lemma}\label{L-B.2}
For each $p\in(0,p_0)$ with some $p_0>0$, a Hermitian matrix $Z(p)$ is given in the
$2\times2$ block form as
$$
Z(p)=\begin{bmatrix}Z_0(p)&Z_2(p)\\Z_2^*(p)&Z_1(p)\end{bmatrix},
$$
where $Z_0(p)$ is $k\times k$, $Z_1(p)$ is $l\times l$ and $Z_2(p)$ is $k\times l$.
Assume:
\begin{itemize}
\item[(a)] $Z_0(p)\to Z_0$ and $Z_2(p)\to Z_2$ as $p\searrow0$,
\item[(b)] there is a $\delta>0$ such that $pZ_1(p)\le-\delta I_l$ for all $p\in(0,p_0)$.
\end{itemize}
Then
$$
e^{Z(p)}\longrightarrow\begin{bmatrix}e^{Z_0}&0\\0&0\end{bmatrix}\quad
\mbox{as $p\searrow0$}.
$$
\end{lemma}

\begin{proof}
We list the eigenvalues of $Z(p)$ in decreasing order (with multiplicities) as
$$
\lambda_1(p)\ge\dots\ge\lambda_k(p)\ge\lambda_{k+1}(p)\ge\dots\ge\lambda_m(p)
$$
together with the corresponding orthonormal eigenvectors
$$
u_1(p),\dots,u_k(p),u_{k+1}(p),\dots,u_m(p),
$$
where $m:=k+l$. Then
\begin{equation}\label{F-B.3}
e^{Z(p)}=\sum_{i=1}^me^{\lambda_i(p)}u_i(p)u_i(p)^*.
\end{equation}
Furthermore, let $\mu_1(p)\ge\dots\ge\mu_k(p)$ be the eigenvalues of $Z_0(p)$ and
$\mu_1\ge\dots\ge\mu_k$ be the eigenvalues of $Z_0$ Then $\mu_i(p)\to\mu_i$ as $p\searrow0$
thanks to assumption (a). By the majorization result for eigenvalues in
\cite[Corollary 7.2]{An} we have
\begin{equation}\label{F-B.4}
\sum_{i=1}^r\mu_i(p)\le\sum_{i=1}^r\lambda_i(p),\qquad1\le r\le k.
\end{equation}
Since
$$
pZ(p)\le\begin{bmatrix}pZ_0(p)&pZ_2(p)\\pZ_2^*(p)&-\delta I_l\end{bmatrix}
\longrightarrow\begin{bmatrix}0&0\\0&-\delta I_l\end{bmatrix}\quad
\mbox{as $p\searrow0$}
$$
thanks to assumptions (a) and (b), it follows that, for $k<i\le m$,
$p\lambda_i(p)<-\delta/2$ for any $p>0$ sufficiently small so that
\begin{equation}\label{F-B.5}
\lim_{p\searrow0}\lambda_i(p)=-\infty,\qquad k<i\le m.
\end{equation}
Hence, it suffices to prove that for any sequence $(p_0>)\ p_n\searrow0$ there exists
a subsequence $\{p_n'\}$ of $\{p_n\}$ such that we have for $1\le i\le k$
\begin{align}
\lambda_i(p_n')&\longrightarrow\mu_i\quad\mbox{as $n\to\infty$}, \label{F-B.6}\\
u_i(p_n')&\longrightarrow v_i\oplus0\in\bC^k\oplus\bC^l\quad\mbox{as $n\to\infty$},
\label{F-B.7}\\
Z_0v_i&=\mu_iv_i. \label{F-B.8}
\end{align}
Indeed, it then follows that $v_1,\dots,v_k$ are orthonormal vectors in $\bC^k$, so from
\eqref{F-B.3} and \eqref{F-B.5} we obtain
$$
\lim_{n\to\infty}e^{Z(p_n')}
=\sum_{i=1}^ke^{\mu_i}v_iv_i^*\oplus0=e^{Z_0}\oplus0.
$$

Now, replacing $\{p_n\}$ with a subsequence, we may assume that $u_i(p_n)$ itself converges
to some $u_i\in\bC^m$ for
$1\le i\le k$. Writing $u_i(p_n)=v_i^{(n)}\oplus w_i^{(n)}$ in $\bC^k\oplus\bC^l$, we have
\begin{align}
\lambda_1(p_n)
&=\bigl\<v_i^{(n)}\oplus w_i^{(n)},Z(p_n)(v_i^{(n)}\oplus w_i^{(n)})\bigr\> \nonumber\\
&=\bigl\<v_i^{(n)},Z_0(p)v_i^{(n)}\bigr\>
+2\Re\bigl\<v_i^{(n)},Z_2(p_n)w_i^{(n)}\bigr\>
+\bigl\<w_i^{(n)},Z_1(p_n)w_i^{(n)}\bigr\> \nonumber\\
&\le\bigl\<v_i^{(n)},Z_0(p_n)v_i^{(n)}\bigr\>
+2\Re\bigl\<v_i^{(n)},Z_2(p_n)w_i^{(n)}\bigr\>
-{\delta\over p_n}\,\big\|w_i^{(n)}\big\|^2 \label{F-B.9}
\end{align}
due to assumption (b). For $i=1$, since $\mu_1(p_n)\le\lambda_1(p_n)$ by \eqref{F-B.4} for
$r=1$, it follows from \eqref{F-B.9} that
$$
p_n\mu_1(p_n)\le p_n\|Z_0(p_n)\|+2p_n\|Z_2(p_n)\|-\delta\big\|w_1^{(n)}\big\|^2,
$$
where $\|Z_0(p_n)\|$ and $\|Z_2(p_n)\|$ are the operator norms. As $n\to\infty$
($p_n\searrow0$), by assumption (a) we have $w_1^{(n)}\to0$ so that
$u_1(p_n)\to u_1=v_1\oplus0$ in $\bC^k\oplus\bC^l$. From \eqref{F-B.9} again we furthermore have
$$
\limsup_{n\to\infty}\lambda_1(p_n)\le\<v_1,Z_0v_1\>
\le\mu_1\le\liminf_{n\to\infty}\lambda_1(p_n)
$$
since $\mu_1(p_n)\le\lambda_1(p_n)$ and $\mu_1(p_n)\to\mu_1$. Therefore,
$\lambda_1(p_n)\to\<v_1,Z_0v_1\>=\mu_1$ and hence $Z_0v_1=\mu_1v_1$. Next, when $k\ge2$ and $i=2$,
since $\lambda_2(p_n)$ is bounded below by \eqref{F-B.4} for $r=2$, it follows
as above that $w_2^{(n)}\to0$ and hence $u_2(p_n)\to u_2=v_2\oplus0$. Therefore,
$$
\limsup_{n\to\infty}\lambda_2(p_n)\le\<v_2,Z_0v_2\>
\le\mu_2\le\liminf_{n\to\infty}\lambda_2(p_n)
$$
so that $\lambda_2(p_n)\to\<v_2Z_0v_2\>=\mu_2$ and $Z_0v_2=\mu_2v_2$, since $\mu_2$
is the largest eigenvalue of $Z_0$ restricted to $\{v_1\}^\perp\cap\bC^k$. Repeating
this argument we obtain \eqref{F-B.6}--\eqref{F-B.8} for $1\le i\le k$.
\end{proof}

Note that the lemma and its proof hold true even when the assumption $Z_2(p)\to Z_2$
in (b) is slightly relaxed into $p^{1/3}Z_2(p)\to0$ as $p\searrow0$. (For this, from
\eqref{F-B.9} note that $p_n^{-1/3}w_i^{(n)}\to0$ and so $Z_2(p_n)w_i^{(n)}\to0$.)

\bigskip\noindent
{\it Proof of Theorem \ref{T-B.1}.}\enspace
Let us divide the proof into two steps. In the proof below we denote by $\triangledown_\alpha$
and $!_\alpha$ the weighted arithmetic and harmonic operator means having the representing
functions $(1-\alpha)+\alpha x$ and $x/((1-\alpha)x+\alpha)$, respectively. Note that
$$
A\,!_\alpha\,B\le A\,\sigma\,B\le A\,\triangledown_\alpha\,B,\qquad A,B\ge0.
$$

\noindent
{\it Step 1.}\enspace
First, we prove the theorem in the case where $P\,\sigma\,Q=P\wedge Q$ for all orthogonal
projections $P,Q$ (this is the case, for instance, when $\sigma$ is the weighted harmonic
operator mean $!_\alpha$, see \cite[Theorem 3.7]{KA}). Let $\cH_0$ be the range of $P_0$
($=A^0\,!_\alpha\,B^0=A^0\,\sigma\,B^0$). From the operator monotonicity of $\log x$ ($x>0$)
it follows that, for every $p>0$,
\begin{equation}\label{F-B.10}
{1\over p}\log(A^p\,!_\alpha\,B^p)\big|_{\cH_0}\le{1\over p}\log(A^p\,\sigma\,B^p)\big|_{\cH_0}
\le{1\over p}\log\bigl(P_0(A^p\,\triangledown_\alpha\,B^p)P_0\bigr)\big|_{\cH_0}.
\end{equation}
For every $\eps>0$ we have
\begin{align*}
(A+\eps A^{0\perp})^p\,!_\alpha\,(B+\eps B^{0\perp})^p
&=\bigl((A+\eps A^{0\perp})^{-p}\,\triangledown_\alpha\,(B+\eps B^{0\perp})^{-p}\bigr)^{-1} \\
&=\bigl(A^{-p}\,\triangledown_\alpha\,B^{-p}
+\eps^{-p}(A^{0\perp}\,\triangledown_\alpha\,B^{0\perp})\bigr)^{-1},
\end{align*}
where $A^{-p}=(A^{-1})^p$ and $B^{-p}=(B^{-1})^p$ are taken as the generalized inverses.
Therefore,
\begin{align}
P_0\bigl((A+\eps A^{0\perp})^p\,!_\alpha\,(B+\eps B^{0\perp})^p\bigr)P_0
&\ge\bigl(P_0\bigl(A^{-p}\,\triangledown_\alpha\,B^{-p}
+\eps^{-p}(A^{0\perp}\,\triangledown_\alpha\,B^{0\perp})\bigr)P_0\bigr)^{-1} \nonumber\\
&=\bigl(P_0(A^{-p}\,\triangledown_\alpha\,B^{-p})P_0\bigr)^{-1}, \label{F-B.11}
\end{align}
since the support projection of $A^{0\perp}+B^{0\perp}$ is
$A^{0\perp}\vee B^{0\perp}=P_0^\perp$. In the above, $(-)^{-1}$ is the generalized
inverse (with support $\cH_0$) and the inequality follows from the operator convexity of
$x^{-1}$ ($x>0$). Letting $\eps\searrow0$ in \eqref{F-B.11} gives
$$
A^p\,!_\alpha\,B^p=P_0(A^p\,!_\alpha\,B^p)P_0
\ge\bigl(P_0(A^{-p}\,\triangledown_\alpha\,B^{-p})P_0\bigr)^{-1}
$$
so that
\begin{equation}\label{F-B.12}
{1\over p}\log(A^p\,!_\alpha\,B^p)\big|_{\cH_0}
\ge-{1\over p}\log\bigl(P_0(A^{-p}\,\triangledown_\alpha\,B^{-p})P_0\bigr)\big|_{\cH_0}.
\end{equation}
Combining \eqref{F-B.10} and \eqref{F-B.12} yields
\begin{equation}\label{F-B.13}
-{1\over p}\log\bigl(P_0(A^{-p}\,\triangledown_\alpha\,B^{-p})P_0\bigr)\big|_{\cH_0}
\le{1\over p}\log(A^p\,\sigma\,B^p)\big|_{\cH_0}
\le{1\over p}\log\bigl(P_0(A^p\,\triangledown_\alpha\,B^p)P_0\bigr)\big|_{\cH_0}.
\end{equation}

Since
$$
A^{-p}=A^0-p\log A+o(p),\qquad B^{-p}=B^0-p\log B+o(p)
$$
as $p\searrow0$, we have
$$
A^{-p}\,\triangledown_\alpha\,B^{-p}
=A^0\,\triangledown_\alpha\,B^0-p((\log A)\,\triangledown_\alpha(\log B))+o(p)
$$
so that
$$
P_0(A^{-p}\,\triangledown_\alpha\,B^{-p})P_0=P_0-p((1-\alpha)\log A\dot+\alpha\log B)+o(p).
$$
Therefore,
\begin{equation}\label{F-B.14}
-{1\over p}\log\bigl(P_0(A^{-p}\,\triangledown\,B^{-p})P_0\bigr)\big|_{\cH_0}
=((1-\alpha)\log A\dot+\alpha\log B)\big|_{\cH_0}+o(1).
\end{equation}
Similarly,
\begin{equation}\label{F-B.15}
{1\over p}\log\bigl(P_0(A^p\,\triangledown\,B^p)P_0\bigr)\big|_{\cH_0}
=((1-\alpha)\log A\dot+\alpha\log B)\big|_{\cH_0}+o(1).
\end{equation}
From \eqref{F-B.13}--\eqref{F-B.15} we obtain
$$
\lim_{p\searrow0}{1\over p}\log(A^p\,\sigma\,B^p)\big|_{\cH_0}
=((1-\alpha)\log A\dot+\alpha\log B)\big|_{\cH_0},
$$
which yields the required limit formula.

\medskip\noindent
{\it Step 2.}\enspace
For a general operator mean $\sigma$ the integral representation theorem
\cite[Theorem 4.4]{KA} says that there are $0\le\theta\le1$, $0\le\beta\le1$ and an operator
mean $\tau$ such that
$$
\sigma=\theta\triangledown_\beta+(1-\theta)\tau
$$
and $P\,\tau\,Q=P\wedge Q$ for all orthogonal projections $P,Q$. Moreover, $\tau$ has the
representing operator monotone function $g$ on $(0,\infty)$ for which $\gamma:=g'(1)\in(0,1)$
and
$$
\alpha=\theta\beta+(1-\theta)\gamma.
$$
We may assume that $0<\theta\le1$ since the case $\theta=0$ was shown in Step 1. Moreover,
when $\theta=1$, we have $\beta=\alpha\in(0,1)$. For the present, assume that $0<\theta\le1$
and $0<\beta<1$. Let $A,B\ge0$ be given, and note that
$A^0\,\sigma\,B^0=\theta A^0\,\triangledown_\beta\,B^0+(1-\theta)(A^0\wedge B^0)$ has the
support projection $A^0\vee B^0$. Let $\cH$, $\cH_0$ and $\cH_1$ denote the ranges of
$A^0\vee B^0$, $P_0=A^0\wedge B^0$ and $A^0\vee B^0-P_0$, respectively, so that
$\cH=\cH_0\oplus\cH_1$. Note that the support of $A^p\,\sigma\,B^p$ for any $p>0$ is $\cH$.
We will describe ${1\over p}\log(A^p\,\sigma\,B^p)\big|_\cH$ in the $2\times2$ block form
with respect to the decomposition $\cH=\cH_0\oplus\cH_1$. Let
$$
Z_0:=((1-\gamma)\log A\dot+\gamma\log B)|_{\cH_0}.
$$
It follows from Step 1 that $\lim_{p\searrow0}(A^p\,\tau\,B^p)^{1/p}=P_0e^{Z_0}P_0$ and
hence
\begin{align*}
A^p\,\tau\,B^p&=P_0\bigl(e^{Z_0}+o(1)\bigr)^pP_0 \\
&=P_0\bigl(I_{\cH_0}+p\log\bigl(e^{Z_0}+o(1)\bigr)+o(p)
\bigr)P_0 \\
&=P_0\bigl(I_{\cH_0}+pZ_0+o(p)\bigr)P_0 \\
&=P_0+p((1-\gamma)\log A\dot+\gamma\log B)+o(p).
\end{align*}
In the above, the third equality follows since
$\log\bigl(e^{Z_0}+o(1)\bigr)=Z_0+o(1)$. On the other hand, we have
$$
A^p\,\triangledown_\beta\,B^p
=A^0\,\triangledown_\beta\,B^0+p((\log A)\,\triangledown_\beta\,(\log B))+o(p).
$$
Therefore, we have
\begin{align*}
A^p\,\sigma\,B^p&=\theta(A^0\,\triangledown_\beta\,B^0)+(1-\theta)P_0 \\
&\qquad+p\theta((\log A)\,\triangledown_\beta(\log B))
+p(1-\theta)((1-\gamma)\log A\dot+\gamma\log B)+o(p).
\end{align*}
Setting
\begin{align*}
C&:=\bigl(\theta(A^0\,\triangledown_\beta\,B^0)+(1-\theta)P_0\bigr)\big|_\cH, \\
H&:=\bigl(\theta((\log A)\,\triangledown_\beta(\log B))
+(1-\theta)((1-\gamma)\log A\,\dot+\gamma\log B)\bigr)\big|_\cH,
\end{align*}
we write
\begin{equation}\label{F-B.16}
{1\over p}\log(A^p\,\sigma\,B^p)\big|_\cH
={1\over p}\log(C+pH+o(p)),
\end{equation}
which $C$ is a positive definite contraction on $\cH$ and $H$ is a Hermitian operator on
$\cH$. Note that the eigenspace of $C$ for the eigenvalue $1$ is $\cH_0$. Hence, with a
basis consisting of orthonormal eigenvectors for $C$ we may assume that $C$ is diagonal so
that $C=\diag(c_1,\dots,c_m)$ with
$$
c_1=\dots=c_k=1>c_{k+1}\ge\dots\ge c_m>0
$$
where $m=\dim\cH$ and $k=\dim\cH_0$.

Applying the Taylor formula (see, e.g., \cite[Theorem 2.3.1]{Hi2} to $\log(C+pH+o(p))$ we
have
\begin{equation}\label{F-B.17}
\log(C+pH+o(p))=\log C+pD\log(C)(H)+o(p),
\end{equation}
where $D\log(C)$ denotes the Fr\'echet derivative of the matrix functional calculus by
$\log x$ at $C$. The Daleckii and Krein's derivative formula (see, e.g.,
\cite[Theorem 2.3.1]{Hi2}) says that
\begin{equation}\label{F-B.18}
D\log(C)(H)=\Biggl[{\log c_i-\log c_j\over c_i-c_j}\Biggr]_{i,j=1}^m\circ H,
\end{equation}
where $\circ$ denotes the Schur (or Hadamard) product and $(\log c_i-\log c_j)/(c_i-c_j)$ is
understood as $1/c_i$ when $c_i=c_j$. We write $D\log(C)(H)$ in the $2\times2$ block form on
$\cH_0\oplus\cH_1$ as $\begin{bmatrix}Z_0&Z_2\\Z_2^*&Z_1\end{bmatrix}$ where
$Z_0:=P_0HP_0|_{\cH_0}$. By \eqref{F-B.16}--\eqref{F-B.18} we can write
$$
{1\over p}\log(A^p\,\sigma\,B^p)={1\over p}\log C+D\log(C)(H)+o(1)
=\begin{bmatrix}Z_0(p)&Z_2(p)\\Z_2^*(p)&Z_1(p)\end{bmatrix},
$$
where
\begin{align*}
Z_0(p)&=Z_0+o(1),\qquad Z_2(p)=Z_2+o(1), \\
Z_1(p)&={1\over p}\,\diag(\log c_{k+1},\dots,\log c_m)+Z_1+o(1).
\end{align*}
This $2\times2$ block form of $Z(p):={1\over p}\log(A^p\,\sigma\,B^p)\big|_\cH$ satisfies
assumptions (a) and (b) of Lemma \ref{L-B.2} for $p\in(0,p_0)$ with a sufficiently small
$p_0>0$. Therefore, the lemma implies that
$$
\lim_{p\searrow0}(A^p\,\sigma\,B^p)^{2/p}\big|_\cH
=\lim_{p\searrow0}\exp\biggl({1\over p}\log(A^p\,\sigma\,B^p)\big|_\cH\biggr)
=e^{Z_0}\oplus0
$$
on $\cH=\cH_0\oplus\cH_1$. Since
\begin{align*}
Z_0=P_0HP_0|_{\cH_0}&=\theta((1-\beta)\log A\dot+\beta\log B)
+(1-\theta)((1-\gamma)\log A\dot+\gamma\log B) \\
&=(1-\alpha)\log A\dot+\alpha\log B,
\end{align*}
we obtain the desired limit formula.

For the remaining case where $0<\theta<1$ and $\beta=0$ or $1$ the proof is similar to the
above when we take as $\cH$ the range of $A^0$ (for $\beta=0$) or $B^0$ (for $\beta=1$)
instead of the range of $A^0\vee B^0$.\qed

\bigskip
Finally, we remark that the same method as in the proof of Step 2 above can also be applied
to give an independent proof of \eqref{F-1.1} for matrices $A,B\ge0$.

\end{document}